\documentclass[reqno]{amsart}

\usepackage[T1]{fontenc}
\usepackage{dsfont,colonequals,mathrsfs}
\usepackage{amsfonts,amssymb,amsmath}
\usepackage{enumerate}
\usepackage{units,url}
\usepackage[american]{babel}

\title{On the Banach-Mazur Type for Normed Spaces}
\author[R.\ Nittka]{Robin Nittka}
\address{University of Ulm\\Institute of Applied Analysis\\89069 Ulm\\Germany}
\email{robin.nittka@uni-ulm.de}
\date{January 4, 2009}


\theoremstyle{plain}
\newtheorem{theorem}{Theorem}[section]
\newtheorem{corollary}[theorem]{Corollary}
\newtheorem{proposition}[theorem]{Proposition}
\newtheorem{lemma}[theorem]{Lemma}

\theoremstyle{definition}
\newtheorem{remark}[theorem]{Remark}

\newtheorem{example}[theorem]{Example}
\newtheorem{definition}[theorem]{Definition}
\newtheorem{question}[theorem]{Question}

\newcommand{\coloneqq}{\colonequals}

\newcommand{\eps}{\varepsilon}
\renewcommand{\rho}{\varrho}
\renewcommand{\phi}{\varphi}
\newcommand{\wto}{\rightharpoonup}

\DeclareMathOperator{\Type}{T}
\DeclareMathOperator{\UType}{UT}
\DeclareMathOperator{\hull}{span}

\begin{document}

\begin{abstract}
	In order to measure qualitative properties we introduce
	a notion of a type for arbitrary normed spaces which measures the
	worst possible growth of partial sums of sequences weakly converging to zero.
	The ideas can be traced back to Banach and Mazur who used this type
	to compare the so-called linear dimension of classical Banach spaces.
	As an application we compare the linear dimension and investigate isomorphy of
	some classical Banach spaces.
\end{abstract}

\maketitle

\section{Introduction}
Since the beginning of functional analysis, a lot of questions have been
asked about the possible structure of Banach spaces, and
it turned out that one can construct spaces having the strangest properties.
Among the most prominent examples are the reflexive Tsirelson space
into which neither an $\ell^p$ space nor $c_0$ can be embedded~\cite{Tsi74},
a space found by Gowers and Maurey not containing an unconditional basic sequence~\cite{GM93},
and the recent positive solution to the
``scalar plus compact problem'' due to S.\ Argyros and R.\ Haydon.
A more detailed list of exotic properties that Banach spaces may exhibit
can be found in the survey articles of Maurey~\cite[Chapter~29]{JL01}
and Gowers~\cite{Gow94}.

Despite or maybe because of the abundance of vastly different Banach spaces
people started to classify the spaces on the basis of some of their most
important characteristics in order
to find some general principles how the various kinds of Banach spaces look like.
Simple criteria to distinguish spaces include qualitative properties such as
reflexivity or separability. More sophisticated investigations also aim towards
qualitative descriptions of the structure, for
example the celebrated Rademacher type and cotype which was first investigated
by Maurey and Pisier in the 1970s.
Soon it turned out that the type and cotype already carry enough information to
characterize the Hilbert spaces among all Banach spaces;
this is a deep result due to Kwapie\'n~\cite{Kwapien72}.
It should be pointed out that the Rademacher type is a local concept, i.e.,
its numerical value for a normed space $X$ depends only on the structure of
the finite dimensional subspaces of $X$.

In the present article we propose a new qualitative measure of the geometry of Banach spaces,
i.e., a new kind of type, which we call the \emph{Banach-Mazur type}.
This is a global concept which was implicitly introduced in Banach's book~\cite[XII.\S 3]{Banach32}
as joint work with Mazur for the special case of $L^p(0,1)$ and $\ell^p$ in order
to investigate the so-called \emph{linear dimension} of those spaces, see
Section~\ref{lineardimension}.
Although it is possible to determine the Banach-Mazur type
of $\ell^p$ by elementary calculations similar for example to the techniques of Banach's
proof of Schur's theorem~\cite[p.~137]{Banach32}, we follow a more elegant approach
based on properties of the canonical Schauder basis of $\ell^p$.
As an immediate application
we are able to prove the well-known fact that $\ell^p$ is not isomorphic to $\ell^q$
for $1 < p < q < \infty$, compare Remark~\ref{ellpellqnoniso}.
Another interesting point is that in contrast to the type and cotype introduced by Maurey and Pisier
the Banach-Mazur type is able to show that for any $p \in (2,\infty)$ 
the spaces $\ell^p$ and $L^p(0,1)$ are not isomorphic.

The article is organized as follows.
In Section~\ref{definition} we define the Banach-Mazur type and the unconditional
Banach-Mazur type. In Section~\ref{properties} we collect some properties which prove to
be useful for the calculation of the type.
Section~\ref{lineardimension} generalizes some ideas of~\cite[XII.\S 3]{Banach32} and shows how this
technique can be used to investigate the so-called \emph{linear dimension} of Banach spaces.
Section~\ref{sums} relates the unconditional Banach-Mazur type of a product of Banach spaces to
the unconditional Banach-Mazur type of its factors. This is interesting mainly because it provides us
with an example where the unconditional Banach-Mazur type is more useful than the Banach-Mazur type.
To actually compute the type of classical spaces, we will utilize basic facts from
the theory of Schauder bases. For the abstract setting this is done in Chapter~\ref{schauderbases}.
Finally, Section~\ref{examples} explicitly determines the Banach-Mazur types for
some classical Banach spaces by applying the results of the preceding sections.
Of course, a lot of interesting problems remain open after this rough investigation
of the Banach-Mazur type. Some questions are collected in Section~\ref{openproblems}.

\section{Definition}\label{definition}
\begin{definition}\label{maindef}
	Let $X$ be a normed vector space, $p \in [1,\infty]$.
	As usual, we set $\frac{1}{\infty} \coloneqq 0$.
	\begin{enumerate}[(i)]
	\item
		We say that
		$X$ is \emph{of Banach-Mazur type $p$} (or simply \emph{of type $p$})
		if the following condition is fulfilled.
		Whenever $(x_n)$ is a sequence in $X$ which converges weakly to $0$
		there exists a subsequence $(x_{n_k})$ of $(x_n)$ and a constant
		$C > 0$ such that $\bigl\| \sum_{k=1}^N x_{n_k} \bigr\| \le C N^{\nicefrac{1}{p}}$
		for all $N \in \mathds{N}$.
	\item\label{maindef:uncond}
		We say that $X$ is \emph{of unconditional Banach-Mazur type $p$}
		(or simply \emph{of unconditional type $p$}) if the following condition
		is fulfilled.
		Whenever $(x_n)$ is a sequence in $X$ which converges weakly to $0$
		there exists a subsequence $(x_{n_k})$ of $(x_n)$ such that for
		every subsequence $(x_{n_{k_\ell}})$ of $(x_{n_k})$ we find a constant $C>0$
		satisfying $\bigl\| \sum_{\ell=1}^N x_{n_{k_\ell}} \bigr\| \le C N^{\nicefrac{1}{p}}$
		for all $N \in \mathds{N}$.
	\item
		We call
		\begin{equation}\label{def:type}
			\Type(X) \coloneqq \sup \left\{ p \in [1,\infty] : X \text{ is of type $p$} \right\}
		\end{equation}
		the \emph{(Banach-Mazur) type of $X$} and
		\begin{equation}\label{def:utype}
			\UType(X) \coloneqq \sup \left\{ p \in [1,\infty] : X \text{ is of unconditional type $p$} \right\}
		\end{equation}
		the \emph{unconditional (Banach-Mazur) type of $X$}.
	\end{enumerate}
\end{definition}

\section{Basic Properties}\label{properties}

Some properties of the type and the unconditional type are immediately obvious.
\begin{enumerate}[(P1)]
\item
	Because every weakly convergent sequence is bounded,
	the sets in~\eqref{def:type} and \eqref{def:utype} are not empty, i.e.,
	$\Type(X) \ge 1$ and $\UType(X) \ge 1$.
\item\label{interval}
	If $X$ is of type $p$, then $X$ is also of type $q$ whenever $1 \le q \le p$,
	and similarly for the unconditional type.
	Thus the sets in~\eqref{def:type} and~\eqref{def:utype} are in
	fact intervals.
\item\label{untypetotype}
	If $X$ is of unconditional type $p$, then $X$ is also of type $p$,
	hence $\UType(X) \le \Type(X)$ for every normed space $X$.
\end{enumerate}

Now we compare the type of spaces which are related by inclusion or isomorphy.
\begin{lemma}\label{monotonicity}
	Let $X$ be a normed space of type $p \in [1,\infty]$ and let $U$ be a subspace of $X$
	equipped with the induced norm. Then also $U$ is of type $p$.
	A similar statement is true for the unconditional type.
	Hence $\Type(U) \ge \Type(X)$ and $\UType(U) \ge \UType(X)$.
\end{lemma}
\begin{proof}
	Let $X$ be of type $p$ and
	let $(x_n)$ be a sequence in $U$ which converges weakly to $0$ in $U$.
	Then $(x_n)$ can also be considered as a sequence in $X$, and it
	is easily checked that it converges to $0$ weakly in $X$.
	Hence by definition $\bigl\| \sum_{k=1}^N x_{n_k} \bigr\| \le C N^{\nicefrac{1}{p}}$
	for some subsequence and some constant $C$.
	The same argument applies to the unconditional type.
\end{proof}

\begin{lemma}\label{completion}
	Let $X$ be a normed space and let $Y$ be
	its completion. Then $X$ is of type $p$ if and
	only if $Y$ is of type $p$.
	The same is true for the unconditional type.
	Thus $\Type(X) = \Type(Y)$ and $\UType(X) = \UType(Y)$.
\end{lemma}
\begin{proof}
	Because $X$ is a subspace of $Y$ one of the implications follows from Lemma~\ref{monotonicity}.
	Now let $X$ be of type $p$. Fix a sequence $(y_n)$ in $Y$ such that
	$y_n \wto 0$ in $Y$. Since $X$ is dense in $Y$ there exist vectors
	$x_n \in X$ such that $\| x_n - y_n \| < \frac{1}{2^n}$.
	Note that then also $x_n \wto 0$ in $X$. Hence by assumption there exist a
	subsequence $(x_{n_k})$ and a constant $C$
	such that $\bigl\| \sum_{k=1}^N x_{n_k} \bigr\| \le C N^{\nicefrac{1}{p}}$.
	For the corresponding subsequence $(y_{n_k})$ we have
	\[
		\Bigl\| \sum_{k=1}^N y_{n_k} \Bigr\|
			\le \Bigl\| \sum_{k=1}^N x_{n_k} \Bigr\| + \sum_{k=1}^N \bigl\| x_{n_k} - y_{n_k} \bigr\|
			\le C N^{\nicefrac{1}{p}} + 1
			\le (C+1) N^{\nicefrac{1}{p}}.
	\]
	This shows that $Y$ is of type $p$.
	The same proof works for the unconditional type.
\end{proof}

The Banach-Mazur type is non-local, i.e., it can not be computed from
its values for the finite dimensional subspaces.
However, it depends only on the structure of the separable subspaces.
We record this observation in the following lemma.
\begin{lemma}\label{separable}
	Let $X$ be a normed space. Then
	\[
		\Type(X) = \inf_{U} \Type(U)
		\quad\text{and}\quad
		\UType(X) = \inf_{U} \UType(U),
	\]
	where $U$ varies over all closed separable subspaces of $X$.
\end{lemma}

\begin{proof}
	It follows from Lemma~\ref{monotonicity} that $\Type(X) \le \inf_{U} \Type(U)$.
	Hence for $T(X) = \infty$ the first equality is proved. Otherwise, i.e., if $T(X) < \infty$,
	let $q > \Type(X)$ be arbitrary. Then there exists a sequence
	$(x_n)$ in $X$ such that $x_n \wto 0$ in $X$ and
	\[
		\sup_{N \in \mathds{N}} \frac{\bigl\| \sum_{k=1}^N x_{n_k} \bigr\|}{N^{\nicefrac{1}{q}}} = \infty
	\]
	for every subsequence $(x_{n_k})$ of $(x_n)$.
	Denote by $V$ the space closure of the space $\hull\{ x_n : n \in \mathds{N}\}$.
	Then $V$ is a closed, separable subspace of $X$ and
	it follows from the Hahn-Banach theorem that $x_n \wto 0$ in $V$.
	By construction $V$ is not of type $q$, hence $\Type(V) \le q$
	according to (P\ref{interval}). This shows
	$\inf_U \Type(U) \le \Type(V) \le q$. Taking the infimum over all $q > T(X)$
	we obtain $\inf_U \Type(U) \le \Type(X)$.
	The proof for the unconditional type is similar.
\end{proof}

\begin{lemma}\label{isomorphyinvariance}
	If two normed spaces $X$ and $Y$ are isomorphic, then $X$ is of
	type $p$ if and only if $Y$ is of type $p$.
	A similar statement holds for the unconditional type.
	In particular $\Type(X) = \Type(Y)$ and $\UType(X) = \UType(Y)$.
\end{lemma}

\begin{proof}
	Let $X$ be of type $p$ and let $J\colon X \to Y$ be an isomorphism.
	Fix a sequence $(y_n)$ satisfying $y_n \wto 0$. Then also
	$x_n \coloneqq J^{-1} y_n \wto 0$. Hence there exists a subsequence
	$(x_{n_k})$ and a constant $C$ such that
	$\bigl\| \sum_{k=1}^N x_{n_k} \bigr\| \le C N^{\nicefrac{1}{p}}$.
	Consequently, for the corresponding subsequence $y_{n_k} = Jx_{n_k}$,
	$\bigl\| \sum_{k=1}^N y_{n_k} \bigr\| \le \|J\| C N^{\nicefrac{1}{p}}$.
	This shows that also $Y$ is of type $p$.
	Since the statement is symmetric with respect to $X$ and $Y$, the claim
	follows for the type. The proof for the unconditional type is
	almost identical.
\end{proof}

\section{Applications to the Linear Dimension}\label{lineardimension}

Banach and Mazur introduced the Banach-Mazur type only for the spaces $\ell^p$ and $L^p$.
They used it as a tool to compare the so-called \emph{linear dimension} of these spaces.
Their ideas can be applied in an abstract situation. As a special case we obtain
their original results. We show Corollary~\ref{incomplindim} as an example,
but the other results about the linear dimensions of $\ell^p$ and $L^p(0,1)$, $p \in (1,\infty)$,
can be obtained in the same way.

Let $X$ and $Y$ be two Banach spaces. We say that \emph{the linear dimension of $X$ does
not exceed the linear dimension of $Y$} if $X$ is isomorphic to a closed subspace of $Y$,
and for this we write $\dim_\ell X \le \dim_\ell Y$.
If neither $\dim_\ell X \le \dim_\ell Y$ nor $\dim_\ell Y \le \dim_\ell X$ holds, we say that
\emph{$X$ and $Y$ are of incomparable linear dimension}.

\begin{lemma}\label{quotienttype}
	Let $X$ and $Y$ be Banach spaces, let $X$ be reflexive, and let $T \in \mathscr{L}(X,Y)$ be onto.
	If $X$ is of type $p \in [1,\infty]$, then also $Y$ is of type $p$.
	A similar statement is true for the unconditional type.
	Hence $\Type(X) \le \Type(Y)$ and $\UType(X) \le \UType(Y)$.
\end{lemma}
\begin{proof}
	The operator $T$ is open according to the open mapping theorem. Hence there exists
	$c > 0$ such that for every $y \in Y$ we find $x \in X$ satisfying $Tx = y$ and
	$\|x\| \le c\|y\|$. Let $(y_n)$ be a sequence in $Y$, $y_n \wto 0$, and choose $(x_n)$
	such that $Tx_n = y_n$ and $\|x_n\| \le c\|y_n\|$. Then $(x_n)$ is bounded in $X$.
	After passing to a subsequence we may therefore assume that $x_n \wto x$.
	Then $y_n = Tx_n \wto Tx$, thus $Tx = 0$.
	Letting $\tilde{x}_n \coloneqq x_n - x$ we have $\tilde{x}_n \wto 0$ and
	$T\tilde{x}_n = y_n$. By assumption we can choose a subsequence of $(\tilde{x}_n)$ such that
	$\bigl\| \sum_{k=1}^N \tilde{x}_{n_k} \bigr\| \le C_X N^{\nicefrac{1}{p}}$ for some constant
	$C_X > 0$. But then $\bigl\| \sum_{k=1}^N y_{n_k} \bigr\| \le C_Y N^{\nicefrac{1}{p}}$
	for $C_Y \coloneqq \|T\| C_X$. This shows that also $Y$ is of type $p$.
	The same reasoning works for the unconditional type.
\end{proof}
\begin{remark}
	In particular, the lemma shows the following:
	Let $Y$ be a quotient of a reflexive Banach space $X$, i.e.,
	$Y$ is isometric to $X/Z$ for a closed subspace $Z$ of $X$.
	Then $\Type(X) \le \Type(Y)$ and $\UType(X) \le \UType(Y)$.
	This is similar to Lemma~\ref{monotonicity}.

	However, the condition of $X$ being reflexive cannot be dropped.
	To see this, recall that every separable Banach space is a quotient
	of $\ell^1$~\cite[Section~1.4]{JL01}, but $\Type(\ell^1) = \infty > 1 = \Type(C[0,1])$,
	see Examples~\ref{ell1} and~\ref{C01}.
\end{remark}

\begin{theorem}\label{dualtype}
	Let $X$ and $Y$ be Banach spaces and $\dim_\ell X \le \dim_\ell Y$.
	If $Y$ is of type $p \in [1,\infty]$, then also $X$ is of type $p$,
	thus $\Type(X) \ge \Type(Y)$.
	If $Y$ is reflexive and $Y'$ is of type $q \in [1,\infty]$, then
	also $X'$ is of type $q$, hence $\Type(X') \ge \Type(Y')$.
	Similar statements hold for the unconditional type.
\end{theorem}
\begin{proof}
	The first assertion follows from Lemmata~\ref{isomorphyinvariance} and~\ref{monotonicity}.
	Now let $Y$ (and hence also $Y'$) be reflexive and let $T\colon X \to Y$ be an isomorphism
	from $X$ onto a closed subspace $Z$ of $Y$.
	Then the adjoint $T'\colon Y' \to X'$ is onto. In fact, for $x' \in X'$ define
	$z' \coloneqq x' \circ T^{-1} \in Z'$ and pick an
	extension $y' \in Y'$ of $z'$ which exists according to the Hahn-Banach theorem.
	Then $T'y' = x'$.
	By Lemma~\ref{quotienttype} this proves the second assertion.
\end{proof}

\begin{corollary}[\mbox{\cite[Th\'eor\`eme~XII.7]{Banach32}}]\label{incomplindim}
	The spaces $\ell^p$ and $\ell^q$
	are of incomparable linear dimension for $p,q \in (1,\infty)$, $p \neq q$.
\end{corollary}
\begin{proof}
	Assume $\dim_\ell \ell^p \le \dim_\ell \ell^q$ for some $p, q \in (1,\infty)$.
	According to Theorem~\ref{dualtype} this implies $\Type(\ell^p) \ge \Type(\ell^q)$
	and $\Type(\ell^{p'}) \ge \Type(\ell^{q'})$. Here $p'$ and $q'$ are the conjugate
	exponents of $p$ and $q$, respectively. By Example~\ref{ellp} this shows
	$p \ge q$ and $\frac{p}{p-1} \ge \frac{q}{q-1}$, implying $p=q$.
\end{proof}

\section{Direct Sums of Banach Spaces}\label{sums}
\begin{proposition}\label{spacesum}
	Let $X_i$ be normed spaces.
	The product $\bigoplus_{i=1}^N X_i$ is of unconditional type $p$
	if and only if all the spaces $X_i$ are of unconditional type $p$.
	In particular, we the unconditional type of the product space is
	$\UType(\bigoplus_{i=1}^N X_i) = \min_i \UType(X_i)$.
\end{proposition}
\begin{proof}
	Without loss of generality we prove the claim only for $N=2$.
	According to Lemma~\ref{isomorphyinvariance} we may choose
	the norm $\| (x, y) \| \coloneqq \| x \|_{X_1} + \| y \|_{X_2}$
	on $X_1 \oplus X_2$. First let $X_1$ and $X_2$ be of unconditional type $p$.
	Let $(x_n, y_n) \wto 0$ in $X_1 \oplus X_2$. Then $x_n \wto 0$ in $X_1$
	and $y_n \wto 0$ in $X_2$.
	We pick subsequences $(x_{n_k})$ of $(x_n)$ and $(y_{n_k})$ of $(y_n)$
	which simultaneously satisfy the growth conditions as in the definition
	of the unconditional type. Then for any subsequence we can estimate
	\[
		\Bigl\| \sum_{\ell=1}^N \bigl( x_{n_{k_\ell}}, y_{n_{k_\ell}} \bigr) \Bigr\|
			= \Bigl\| \sum_{\ell=1}^N x_{n_{k_\ell}} \Bigr\| + \Bigl\| \sum_{\ell=1}^N y_{n_{k_\ell}} \Bigr\|
			\le 2C N^{\nicefrac{1}{p}}.
	\]
	Thus $X_1 \oplus X_2$ is of unconditional type $p$.

	If, on the other hand, $X_1 \oplus X_2$ are of unconditional type $p$,
	then also its subspaces $X_1$ and $X_2$ are of unconditional type $p$
	according to Lemma~\ref{monotonicity}.
\end{proof}

\begin{remark}
	The above argument does not work for the type because it is not obvious whether we can pick subsequences
	which simultaneously satisfy the growth conditions. However, if $X_1$ is of
	type $p_1$ and $X_2$ is of unconditional type $p_2$, then
	$X_1 \oplus X_2$ is of type $\min\{ p_1, p_2 \}$, but not of type $q$ for any
	$q > \min\{ p_1, p_2 \}$. The proof is essentially the same as the above.
\end{remark}

\begin{proposition}
	Let $X_i$ be normed subspaces of unconditional type $p$ which are contained in
	a common vector space $Z$. Then the normed space
	$U \coloneqq \bigcap_{i=1}^N X_i$ equipped with the usual norm 
	$\| x \|_U \coloneqq \sum_{i=1}^N \| x \|_{X_i}$ is of unconditional
	type $p$. In particular, $\UType(\bigcap_i X_i) \ge \min_i \UType(X_i)$.
\end{proposition}
\begin{proof}
	The space $\bigoplus_i X_i$ is of unconditional type $p$ by Proposition~\ref{spacesum}.
	Since $\bigcap_i X_i$ can be identified with the diagonal in $\bigoplus_i X_i$,
	Lemma~\ref{monotonicity} shows that also $\bigcap_i X_i$ is of unconditional type $p$.
\end{proof}

\begin{remark}
	The above estimate can in general not be improved.
	In fact, if $\bigcap_i X_i = \{0\}$,
	then $\UType(\bigcap_i X_i) = \infty$ independently of the unconditional types of the $X_i$.
\end{remark}

\section{Some Relations to Schauder bases}\label{schauderbases}
The following proposition can be applied to calculate the type of a whole class of Banach spaces.
We will use it to determine the type of $\ell^p$ and $c_0$.
Recall that a sequence $(x_n)$ in a Banach space $X$ is called a \emph{Schauder basis}
of $X$ if every $x \in X$ has a unique representation as a convergent series
$x = \sum_i \alpha_i x_i$. A sequence $(y_n)$ is called a \emph{basic sequence}
if it is a Schauder basis of $\overline{\hull}\{ y_n : n \in \mathds{N} \}$.
A basic sequence $(y_n)$ is called a \emph{block basic sequence of the basic sequence $(y_n)$}
if it can be written in the form $u_n = \sum_{i=m_n}^{m_{n+1}-1} \alpha_i y_i$ with a strictly increasing
sequence $(m_n)$ of natural numbers.
We call two basic sequences $(x_n)$ and $(y_n)$ \emph{equivalent} if
$Sx_n = y_n$ extends to an isomorphism between the spans of the sequences.

\begin{proposition}\label{blockbasis}
	Let $p \in [1,\infty]$ and $X$ be a Banach space.
	Let $(e_n)$ be a Schauder basis and assume that for every bounded block basic sequence
	$(y_\ell)$ of $(e_n)$ there exists a constant $c$ such that
	$\bigl\| \sum_{\ell=1}^N y_\ell \bigr\| \le c N^{\nicefrac{1}{p}}$ for all $N$.
	Then $X$ is of unconditional type $p$.
\end{proposition}

\begin{proof}
	Let $(x_n)$ be a sequence in $X$ such that $x_n \wto 0$. We have to show
	that there exists a subsequence $(x_{n_k})$ such that every subsequence
	$(x_{n_{k_\ell}})$ of $(x_{n_k})$ we find a constant $C>0$ satisfying
	\begin{equation}\label{desestimate}
		\Bigl\| \sum_{\ell=1}^N x_{n_{k_\ell}} \Bigr\| \le C N^{\nicefrac{1}{p}} \text{ for all } N \in \mathds{N}.
	\end{equation}
	This is trivial if $\liminf_n \| x_n \| = 0$; in this case
	we may pick $C \coloneqq 1$ for a subsequence $(x_{n_k})$ such that $\| x_{n_k} \| \le 2^{-k}$.
	So without loss of generality (passing to a subsequence if necessary) we may
	assume that there are constants such that $0 < \eps_0 \le \| x_n \| \le M$ for all $n$.
	Now it is possible to pick a subsequence $(x_{n_k})$ of $(x_n)$ which is
	equivalent to a bounded block basic sequence $(y_k)$ of $(e_n)$,
	see~\cite[Proposition~1.a.12]{Sequences77}.
	In particular, the linear mapping $S\colon \hull\{ y_k \} \to X$ defined by
	$Sy_k \coloneqq x_{n_k}$ is bounded.
	Let $(x_{n_{k_\ell}})$ be any subsequence of $(x_{n_k})$. Then $(x_{n_{k_\ell}})$
	is equivalent to the bounded block basic sequence $(y_{k_\ell})$ of $(e_n)$.
	From the assumption we then obtain \eqref{desestimate} for $C \coloneqq c \left\|S\right\|$.
	This concludes the proof.
\end{proof}

The following lemma describes a situation in which we can assure that the type
and the unconditional type coincide.

\begin{proposition}\label{typeequncond}
	Let $p \in [1,\infty]$.
	Assume that $X$ is a Banach space with the following property.
	Every basic sequence $(x_n)$ in $X$ satisfying $x_n \wto 0$ and $\liminf \| x_n \| > 0$
	possesses a subsymmetric subsequence $(x_{n_k})$, i.e.,
	every subsequence $(x_{n_{k_\ell}})$ of $(x_{n_k})$ is equivalent to $(x_{n_k})$.
	Then $X$ is of type $p$ if and only if $X$ is of unconditional type $p$.
\end{proposition}
\begin{proof}
	One implication is true in general, see (P\ref{untypetotype}).
	So now let $X$ be of type $p$ and let $(x_n)$ be such that $x_n \wto 0$.
	As is the proof of Proposition~\ref{blockbasis} we may without loss of generality
	assume that $\liminf \| x_n \| > 0$.
	By passing to a subsequence if necessary we can achieve that $(x_n)$ is basic sequence,
	see the remarks following the proof of~\cite[Theorem~1.a.5]{Sequences77}.
	By assumption there exists a subsymmetric subsequence $(x_{n_k})$ of $(x_n)$.
	We claim that $(x_{n_k})$ has the property which Definition~\ref{maindef}.\eqref{maindef:uncond} requests.
	Since $X$ is of type $p$ we can find a subsequence $(x_{n_{\tilde{k}_\ell}})$ of $(x_{n_k})$
	and a constant $C > 0$ such that
	\begin{equation}\label{subsymmest}
		\Bigl\| \sum_{\ell=1}^N x_{n_{\tilde{k}_\ell}} \Bigr\| \le C N^{\nicefrac{1}{p}}
	\end{equation}
	holds for all $N \in \mathds{N}$.
	By subsymmetry $(x_{n_{\tilde{k}_\ell}})$ is equivalent to $(x_{n_k})$.
	Now if $(x_{n_{k_\ell}})$ is any subsequence of $(x_{n_k})$,
	by assumption $(x_{n_{k_\ell}})$ is equivalent to $(x_{n_k})$
	and hence also to $(x_{n_{\tilde{k}_\ell}})$.
	Thus, as in the proof of Proposition~\ref{blockbasis},
	\eqref{subsymmest} remains valid for $(x_{n_{k_\ell}})$ if we multiply
	$C$ by the operator norm of the corresponding isomorphism.
\end{proof}
\begin{example}
	The spaces $c_0$ and $\ell^p$, $1 \le p < \infty$ satisfy the assumptions
	of the Lemma~\ref{typeequncond}. To see this, let $(e_n)$ be the usual Schauder
	basis of the respective space and let $(x_n)$ be as in the lemma.
	We can find a subsequence $(x_{n_k})$ of $(x_n)$ which is equivalent to a block
	basic sequence of $(e_n)$, see~\cite[Proposition~1.a.12]{Sequences77}.
	Then $(x_{n_k})$ is equivalent to $(e_n)$~\cite[Proposition~2.a.1]{Sequences77}
	and hence subsymmetric.
\end{example}

\section{Examples}\label{examples}

First we consider a class of spaces whose type is particularly easy to compute.
A normed space $X$ is called a \emph{Schur space} if
every weakly convergent sequence is convergent in norm.
Examples of Schur spaces are finite dimensional spaces and $\ell^1$,
cf.~\cite[p.~137]{Banach32}.
\begin{example}\label{ell1}
	Every Schur space $X$ is of unconditional type $\infty$, hence
	$\UType(X) = \Type(X) = \infty$.
\end{example}
\begin{proof}
	Let $(x_n)$ be a sequence in $X$ such that $x_n \wto 0$.
	Then $x_n \to 0$. Thus there exists a subsequence
	$(x_{n_k})$ such that $\| x_{n_k} \| \le 2^{-k}$.
	In particular, $\sum_{k=1}^N x_{n_{k_\ell}}$ is bounded in $X$
	for every subsequence $(x_{n_{k_\ell}})$ of $(x_{n_k})$.
	This shows that $X$ is of unconditional type $\infty$.
	Using (P\ref{untypetotype}), the claim is proved.
\end{proof}

\begin{example}\label{ellp}
	For every $p \in (1,\infty)$ the space $\ell^p$ is of unconditional type $p$,
	but not of type $q$ for any $q > p$, hence $\UType(\ell^p) = \Type(\ell^p) = p$.
	Moreover, $c_0$ is of unconditional type $\infty$, hence
	$\UType(c_0) = \Type(c_0) = \infty$.
\end{example}
\begin{proof}
	Consider the sequence $(e_n)$ of canonical unit vectors in $\ell^p$.
	It is easy to see that $e_n \wto 0$ and that for any subsequence
	$(e_{n_k})$ we obtain
	\[
		\Bigl\| \sum_{k=1}^N e_{n_k} \Bigr\|_p
			= N^{\nicefrac{1}{p}}
		\quad \text{for all $N \in \mathds{N}$.}
	\]
	This shows that $\ell^p$ is not of type $q$ for any $q > p$.

	To show that $\ell^p$ is of unconditional type $p$ and that $c_0$
	is of unconditional type $\infty$ we apply Proposition~\ref{blockbasis}
	to the Schauder basis $(e_n)$ of canonical unit vectors.
	First consider the case of $\ell^p$.
	For any bounded block basic sequence $(y_\ell)$ of $(e_n)$ we obtain
	$\bigl\| \sum_{\ell=1}^N y_\ell \bigr\|_p^p = \sum_{\ell=1}^N \| y_\ell \|_p^p \le M^p N$
	by an easy calculation, where $M$ denotes an upper bound for $\|y_n\|_p$.
	Similarly, for a bounded block basic sequence $(y_\ell)$ in $c_0$ we have
	$\bigl\| \sum_{\ell=1}^N y_\ell \bigr\|_\infty = \max_\ell \| y_\ell \|_\infty \le M$,
	where $M$ denotes an upper bound for $\|y_n\|_\infty$.
	Now Proposition~\ref{blockbasis} asserts that $\ell^p$ is of unconditional type $p$ and
	$c_0$ is of unconditional type $\infty$.
	The remaining assertions follow from (P\ref{interval}) and (P\ref{untypetotype}).
\end{proof}

\begin{remark}\label{ellpellqnoniso}
	Example~\ref{ellp} provides a proof (via Lemma~\ref{isomorphyinvariance})
	of the fact that the sequence spaces $\ell^p$ and $\ell^q$
	are not isomorphic for $p \neq q$.
	Of course this is well-known and follows also from Pitt's Theorem~\cite[\S 42.3.(10)]{Koethe79} or
	the calculation of the Rademacher types and cotypes of these spaces~\cite{Schwartz81}.
\end{remark}

\begin{example}
	Let $I$ be an arbitrary infinite index set. We define $\ell^p(I)$ for $1 \le p < \infty$ as usual,
	i.e., $\ell^p(I)$ is the completion of the space of families $x = (x_\alpha)_{\alpha \in I}$
	where only finitely many components $x_\alpha$ do not vanish, with respect to the norm
	$\| x \|_p \coloneqq \sum_{\alpha \in I} |x_\alpha|^p$.
	Then $\Type(\ell^p(I)) = \UType(\ell^p(I)) = p$.
	In particular, $\ell^p(I_1)$ and $\ell^q(I_2)$ are not isomorphic if $p \neq q$
	and either one of $I_1$ and $I_2$ is infinite.
\end{example}
\begin{proof}
	Since $\ell^p$ can naturally considered as a closed subspace of $\ell^p(I)$,
	Lemma~\ref{monotonicity} ensures that $\Type(\ell^p(I)) \le p$.
	If on the other hand $(x_n)$ is any sequence in $\ell^p(I)$ such that $x_n \wto 0$ in $\ell^p(I)$,
	then the set $J$ of indices where at least one of the $x_n$ does not vanish, is at most countable.
	So $(x_n)$ is a sequence in the subspace $\ell^p(J)$ of $\ell^p(I)$ and $x_n \wto 0$ in $\ell^p(J)$.
	Because $\ell^p(J)$ is isomorphic to $\ell^p$ or a finite dimensional space,
	the estimate on the partial sums follows either from Example~\ref{ell1} or Example~\ref{ellp}.
	This shows $\UType(\ell^p(I)) \ge p$.
\end{proof}

\begin{example}[Hilbert spaces]
	Every infinite dimensional inner product space $H$ is of unconditional type $2$,
	but not of type $q$ for any $q > 2$, hence $\UType(H) = \Type(H) = 2$.
\end{example}
\begin{proof}
	The case $H = \ell^2$ is a special case of Example~\ref{ellp}.
	Thus the claim holds for every infinite dimensional separable Hilbert space due to
	Lemma~\ref{isomorphyinvariance}.
	Now the assertion follows from Lemmata~\ref{completion} and~\ref{separable}.
\end{proof}

So far we have only seen examples of Banach spaces $X$ with $\Type(X) = \UType(X) = p$
for $p \in (1,\infty]$. The following example shows that also $p=1$ is possible.
\begin{example}\label{C01}
	$\Type(C[0,1]) = \UType(C[0,1]) = 1$.
\end{example}
\begin{proof}
	In fact, it is known that every separable Banach space $X$ is
	isometric to a closed subspace of $C[0,1]$~\cite[Section~1.4]{JL01}. Hence
	by Lemma~\ref{separable}
	\[
		\UType(C[0,1]) \le \Type(C[0,1]) \le \inf_{p \in (1,\infty)} \Type(\ell^p) = 1. \qedhere
	\]
\end{proof}

Recall that the Rademacher cotype of $\ell^p$ and $L^p(0,1)$ is $p$ for every $2 < p < \infty$.
For the Banach-Mazur type the situation is different as the following example shows.
\begin{example}\label{Lptype}
	Let $2 < p < \infty$. Then $\UType(L^p(0,1)) = \Type(L^p(0,1)) = 2$.
\end{example}
\begin{proof}
	It can be shown that $L^p(0,1)$ contains an isomorphic copy of $\ell^2$~\cite[Section~1.4]{JL01},
	so by Lemma~\ref{monotonicity} the type of $L^p(0,1)$ can be at most $2$.
	But on the other hand, if $(f_n)$ is a sequence in $L^p(0,1)$ such that
	$0 < \limsup_n \|f\|_n < \infty$, then $(f_n)$ has a subsequence which is equivalent
	to the unit vector sequence in either $\ell^2$ or $\ell^p$~\cite[Section~1.8]{JL01}. In both cases
	we obtain the estimate for the unconditional type, thus showing $\UType(L^p(0,1)) \ge 2$,
	and the assertion is proved.
\end{proof}
\begin{remark}
	It is possible to show that in fact $\Type(L^p(0,1)) = p$ for $p \in (1,2]$, see~\cite[Th\'eor\`eme~XII.2]{Banach32}.
\end{remark}
This shows that the spaces $\ell^p$ and $L^p(0,1)$ are non-isomorphic at least for $p \in (2,\infty)$
which does not follow from the value of the Rademacher types of these spaces.
Thus the Banach-Mazur type carries interesting information about a Banach space
which escapes some other concepts.

\section{Open Problems}\label{openproblems}
\begin{question}
	Does any Banach space $X$ satisfy $\UType(X) = \Type(X)$?
\end{question}
\begin{question}
	If $X$ is of type $p \coloneqq \Type(X) < \infty$, is there a sequence $(x_n)$ in $X$
	such that any subsequence $(x_{n_k})$ of $(x_n)$ satisfies
	\[
		\limsup_{N \to \infty} \frac{\bigl\| \sum_{k=1}^N x_{n_k} \bigr\|}{N^{\nicefrac{1}{p}}} > 0,
	\]
	and similarly for the unconditional type? In other words, is there always a sequence
	for which the growth estimate is sharp?
\end{question}
\begin{question}
	Can the constant $C$ in~\eqref{def:type} and~\eqref{def:utype} always be chosen as
	\[
		C \le c_p \limsup_{n \to \infty} \| x_n \|,
	\]
	where $c_p$ is some universal constant depending only on $p$?
\end{question}

\section*{Acknowledgements}
The author expresses his gratitude to Prof.~Pietsch for valuable hints regarding the Banach-Mazur type,
in particular for pointing out that this concept has already been investigated
by Banach and Mazur.

The author thanks the Graduate School \textsl{Mathematical Analysis of Evolution, Information and Complexity}
for their support during the work on this article.

\bibliographystyle{amsalpha}
\bibliography{type}

\providecommand{\bysame}{\leavevmode\hbox to3em{\hrulefill}\thinspace}
\providecommand{\MR}{\relax\ifhmode\unskip\space\fi MR }
\providecommand{\MRhref}[2]{%
  \href{http://www.ams.org/mathscinet-getitem?mr=#1}{#2}
}
\providecommand{\href}[2]{#2}
\begin{thebibliography}{Gow94}

\bibitem[Ban32]{Banach32}
S.~Banach, \emph{{Th{\'e}orie des Op{\'e}rations Lin{\'e}aires}}, Monografie
  Matematyczne, 1932.

\bibitem[GM93]{GM93}
W.T. Gowers and B.~Maurey, \emph{{The unconditional basic sequence problem}},
  J. Amer. Math. Soc. \textbf{6} (1993), 851--874.

\bibitem[Gow94]{Gow94}
W.T. Gowers, \emph{{Recent results in the theory of infinite-dimensional Banach
  spaces}}, Proc. ICM Z\"urich, 1994, pp.~933--942.

\bibitem[JL01]{JL01}
W.B. Johnson and J.~Lindenstrauss, \emph{{Handbook of the Geometry of Banach
  Spaces}}, North Holland, 2001.

\bibitem[K{\"o}t79]{Koethe79}
G.~K{\"o}the, \emph{{Topological Vector Spaces II}}, Springer, 1979.

\bibitem[Kwa72]{Kwapien72}
S.~Kwapie\'n, \emph{{Isomorphic characterizations of inner product spaces by
  orthogonal series with vector valued coefficients}}, Studia Math \textbf{44}
  (1972), 583--595.

\bibitem[LT77]{Sequences77}
J.~Lindenstrauss and L.~Tzafriri, \emph{{Classical Banach Spaces I: Sequence
  Spaces}}, Springer, 1977.

\bibitem[Sch81]{Schwartz81}
L.~Schwartz, \emph{{Geometry and Probability in Banach Spaces}}, American
  Mathematical Society \textbf{4} (1981), no.~2, 135--141.

\bibitem[Tsi74]{Tsi74}
B.S. Tsirel'son, \emph{{Not every Banach space contains an imbedding of
  {$\ell_p$} or {$c_0$}}}, Functional Analysis and Its Applications \textbf{8}
  (1974), no.~2, 57--60.

\end{thebibliography}

\end{document}